\documentclass[10pt,leqno]{amsart}%

\usepackage{latexsym}
\usepackage{amsfonts}
\usepackage{amsmath}
\usepackage{amsfonts}
\usepackage{amssymb}
\usepackage{amscd, upgreek}
\usepackage{amsfonts}
\usepackage{amsbsy}
    {}
\usepackage{amssymb}

\newcommand{\diver}{{{\rm{div}}}}

\newcommand{\OM}{{\Omega}}

\newtheorem{theorem}{Theorem}

\newtheorem{corollary}[theorem]{Corollary}

\newtheorem{definition}[theorem]{Definition}

\newtheorem{lemma}[theorem]{Lemma}

\newtheorem{remark}[theorem]{Remark}

\newtheorem{assunz}[theorem]{Assumption}

\newcommand{\Hess}{{\rm Hess\,}}
\newcommand{\disp}{\displaystyle}
\newcommand{\erre}{\mathbb{R}}

\newcommand{\esse}{\mathbb{S}}

\newcommand{\di}{\mathrm{d}} % per i differenziali

\newcommand{\ra}{\rightarrow}

\newcommand{\pair}[1]{\langle#1\rangle}      % pairing

\newcommand{\eps}{\varepsilon}                           % epsilon
\newcommand{\metric}{\pair{\;,}}                          % metrica
\renewcommand{\tilde}[1]{\widetilde{#1}}

\newcommand{\ess}{\mathrm{ess}}
\newcommand{\rad}{\mathrm{rad}}

\begin{document}

\title[Eigenvalue estimates for submanifolds]
  {Eigenvalue estimates for submanifolds of
warped product spaces}

\author{G. P.  Bessa}
\address{Departamento de Matem\'atica \\Universidade Federal do Cear\'a-UFC\\
60455-760 Fortaleza, CE, Brazil}
\email{bessa@mat.ufc.br}
\author{S.C. Garc\'ia-Mart\'inez}
\address{ Departamento de Matem\'aticas,\\ Universidad de Murcia, \\ Campus de Espinardo, 30100 Espinardo, Murcia, Spain}
\email{sandracarolina.garcia@um.es}
\author{L. Mari }
\address{Departamento de Matem\'atica \\Universidade Federal do Cear\'a-UFC\\
60455-760 Fortaleza, CE, Brazil}
\email{lucio.mari@libero.it}
\author{H.F. Ramirez-Ospina}
\address{ Departamento de Matem\'aticas,\\ Universidad de Murcia, \\ Campus de Espinardo, 30100 Espinardo, Murcia, Spain}
\email{hectorfabian.ramirez@um.es}

%\volume{} \pubyear{} \setcounter{page}{1} \receivedline{Received
%\textup{} month \textup{}
%              revised \textup{} month \textup{}}
\maketitle
\begin{abstract}
In this paper, we give lower bounds for the fundamental tone of open sets in minimal submanifolds immersed into warped product spaces of type $N^n \times_f Q^q$, where $f \in C^\infty(N)$. Some applications, also regarding the essential spectrum, illustrate the applicability and the generality of our results.
\end{abstract}

\section{Introduction}
Let $M$ be a connected Riemannian manifold, possibly incomplete, and let $\Delta=\diver \circ \nabla$ be the Laplace-Beltrami operator on acting on $C^\infty_{o}(M)$,  the space of smooth functions with compact support. %Among all the possible self-adjoint extensions of $\Delta$, we will always consider the Friedrichs extension, also denoted by $\Delta$.
When $M$ is geodesically complete, $\Delta$ is essentially self-adjoint, thus there is a unique self-adjoint extension to  an unbounded operator, denoted by $\Delta$, whose domain is the set of  functions $ f\in L^{2}(M)$ so that $\Delta\! f\in L^{2}(M)$, see
 \cite{chernoff}, \cite{davies} and \cite{strichartz}. If $M$ is not complete  we will always  consider the Friedrichs extension of $\Delta$.  Denote by $\sigma(-\Delta)$ and $\sigma_\ess(-\Delta)$, respectively, the spectrum and the essential spectrum of $-\Delta$. Given an open subset $\OM \subset M$, the fundamental tone of $\OM$, $\lambda^{\ast}(\OM)$, is defined by
\[\lambda^{\ast}(\OM)= \inf \sigma(-\Delta) = \inf\left\{\frac{\smallint_{\OM}\vert
\nabla f \vert^{2}}{\smallint_{\OM} f^{2}};\, f\in
 H_{0}^1(\Omega) \backslash \{ 0\}\right\}.
\]
When $\OM$ has compact closure and Lipschitz boundary, $\lambda^{\ast}(\OM)$ coincides with the
first eigenvalue $\lambda_1(\OM)$ of $\OM$, with Dirichlet boundary data
on $\partial \OM$. Its associated eigenspace is $1$-dimensional and spanned by any solution $u$ of
$$
\left\{\begin{array}{l}
\Delta u + \lambda_1(\Omega)u =  0  \qquad \text{on } \, \Omega,\\[0.2cm]
u=0 \qquad \text{on } \, \partial \Omega.
\end{array}\right.
$$
The relations between the fundamental tone of open sets of $M$ and their geometric invariants has been the subject to an intensive research in the past 50 years. Among a huge literature, we limit ourselves to quote the classics  \cite{berard}, \cite{berger-gauduchon-mazet}, \cite{chavel} and references therein for a detailed picture.  In particular,  a great effort has been done to estimate the fundamental tone of minimal submanifolds of well-behaved ambient spaces (for instance, in \cite{bessa-montenegro1}, \cite{bessa-montenegro2}, \cite{candel}, \cite{cheng-li-yau} and \cite{cheung-leung}). In this paper, we move a step further by giving lower bounds for the fundamental tone of manifolds which are minimally immersed in ambient spaces  $N^n\times_f Q^q$ carrying a warped product structure, see Theorem \ref{thm2} below. As we shall see in the last section, the generality of our setting allows applications to submanifolds of cylinders, cones, tubes, improving  certain recent results in the literature (\cite{BS}, \cite{bessa-montenegro1}, \cite{bessa-montenegro2}). We remark that there have been an increasing interest in the study of minimal and constant mean curvature submanifolds in product spaces $N\times \mathbb{R}$, after the discovery of many beautiful examples such as those in \cite{meeks-rosenberg}, \cite{meeks-rosenberg2}, and this motivates a thorough investigation of the spectrum of such submanifolds. In this respect, we hope that our estimates could be useful.

\section{Preliminaries}

\noindent \textbf{Isometric immersions}\\
Let $M$ and $W$ be smooth %complete
Riemannian manifolds of dimension
 $m$ and $n+q$  respectively and $\varphi \colon M \hookrightarrow W$ be an isometric immersion.  Consider a smooth function $F:W \rightarrow \mathbb{R}$ and the composition $F\circ
 \varphi \colon M \rightarrow \mathbb{R}$. Identifying $X$ with $d\varphi (X)$, the Hessian of $F\circ \varphi$ at $x\in M$ is given by
 \begin{equation}\Hess_{M}(F\circ
 \varphi) (x)  \,(X,Y)= \Hess_{W}F(\varphi (x))\,(X,Y) +
 \langle \nabla F,\,\sigma (X,Y)\rangle_{\varphi (x)},
\label{eqBF2}
\end{equation}
where $\sigma(X,Y)$ is the second fundamental form of $\varphi$. Tracing (\ref{eqBF2}) with respect to an orthonormal basis $\{ e_{1},\ldots e_{m}\}$,
\begin{align}
\Delta_{M} (F\circ
 \varphi) (x) & =\sum_{i=1}^{m}\left\{\Hess_{W}F (\varphi
(x))\,(e_{i},e_{i}) + \langle \nabla F,\, \sum_{i=1}^{m}\sigma
(e_{i},e_{i})\rangle\right\}\nonumber\\
&=\sum_{i=1}^{m}\Hess_{W}F (\varphi
(x))\,(e_{i},e_{i}) + m\langle \nabla F,H\rangle,\label{eqBF3}
\end{align}
where  $H = m^{-1}\mathrm{tr}(\sigma)$ is the normalized mean curvature vector. Formulae \eqref{eqBF2} and \eqref{eqBF3} are well known in
the literature, see \cite{jorge-koutrofiotis}.\\
\par
\noindent \textbf{Models and Hessian comparisons}\\
Hereafter, we denote with $\erre^+_0= [0,+\infty)$. Let $g\in C^2(\mathbb{R}_0^+)$ be  positive in $(0, R_0)$, for some $0<R_0\leq \infty$, and satisfying
$$
g(0)=0, \qquad g'(0)=1.
$$
The $\kappa$-dimensional model manifold $\mathbb{Q}_g^{\kappa}$ constructed from the function $g$ is the ball $B_R(o) \subseteq \erre^\kappa$ with metric given, in polar geodesic coordinates centered at $o$, by
\[\di s^2_g=\di r^2+g(r)^2\left\langle\, ,\right\rangle_{\mathbb{S}^{\kappa-1}},\] where
$\left\langle\, ,\right\rangle_{\mathbb{S}^{\kappa-1}}$ is the standard metric on the unit $(\kappa-1)$-sphere. The radial sectional curvature and the Hessian of the distance function $r$ on $\mathbb{Q}^\kappa_g$ are given by the expressions
$$
K^\rad = - \frac{g''(r)}{g(r)}, \qquad \Hess r = \frac{g'(r)}{g(r)} \Big( \di s^2 - \di r \otimes \di r \Big).
$$

From the first relation, we see that a model can equivalently be specified by prescribing its radial sectional curvature $G \in C^\infty(\erre^+_0)$ and recovering $g$ as the solution of
\begin{align}\label{eqg}
\left\{
  \begin{array}{l}
   g''-Gg=0, \\[0.1cm]
   g(0)=0, \quad g'(0)=1,
  \end{array}
\right.
\end{align}
on the maximal interval $(0, R_0)$ where $g>0$.\par
\vspace{2mm}

For the proof of our main results we will make use of the following version of the Hessian Comparison Theorem, see \cite{GW} and  \cite[Chapter 2]{70}.
\begin{theorem}\label{hessiancomparison} Let $Q^q$ be a complete Riemannian $q$-manifold. Fix a  point $o\in Q$, denote by $\rho_{_Q}(x)$ the Riemannian distance function from $o$ and let $D_o=Q\backslash\text{cut}(o)$ be the domain
of the normal geodesic coordinates centered at $o$. Given $G \in C^\infty(\erre^+_0)$, let $g$ be the solution of the Cauchy problem \eqref{eqg}, and let $(0,R_0)\subseteq [0,+\infty)$ be the maximal interval where $g$ is positive. If the radial sectional curvature of $Q$ satisfies
\begin{align}
\label{1.30.bis}
 K^{\emph{rad}}_{Q}\leq - G(\rho_{_Q}) \quad (\textrm{respectively, } \,\, \,K^{\emph{rad}}_{Q}\geq - G(\rho_{_Q}) ),
\end{align}
on $B(o,R_0)$, then
\[
\Hess_{_Q} \rho_{_Q}\!\geq\!\frac{g'(\rho_{_Q})}{g(\rho_{_Q})}\left(\left\langle \,,\,\right\rangle_Q\!\!-d\rho_{_Q}\otimes d\rho_{_Q}\right)  \,\,(\text{respectively, } %\Hess \rho\!\geq\!\frac{g'(\rho)}{g(\rho)}\left(\left\langle \,,\,\right\rangle_Q\!\!-d\rho\otimes d\rho\right)\!
\leq )
\]
on $D_o\cap B(o,R_0)\backslash\{ o\}$, in the sense of quadratic forms.

%If we assume, instead of \eqref{1.30.bis}, that the radial Ricci curvature of $Q$ satisfies
%\begin{align}
%\label{1.32.bis}
%\emph{Ric}(\nabla \rho,\nabla \rho)\geq -(q-1)G(\rho)
%\end{align}
%for some smooth non-negative even function $G$ on $[0,+\infty)$, then the inequality
%\[
%\Delta \varrho\leq (p-1)\frac{u'(\varrho)}{u(\varrho)}
%\]
%holds pointwise on $P\backslash(\text{cut}(o)\cup\{ o\})$, and weakly on all of $P$, where $u$ is a solution of the problem
%\begin{align}
%\left\{
%  \begin{array}{l}
%   u''(t)-G(t)u(t)\geq 0, \\
%   u(0)=0, \quad u'(0)=1.
%  \end{array}
%\right.
%\end{align}
\end{theorem}

\noindent \textbf{Eigenvalues and Eigenfunctions}\\
The generalized version of Barta's Eigenvalue Theorem \cite{barta}, %(see , \cite{chengyau} and also
 proved in \cite{bessa-montenegro2} will be important in the sequel. %For the convenience of the reader, we add a short proof. (CHECK THAT THE PROOF IS DIFFERENT FROM THE ONE IN \cite{chengyau}. IF NOT, YOU CAN ERASE IT, IT IS NOT NECESSARY)
\begin{theorem}%[\cite{bessa-montenegro2}]
\label{thmBM1}Let
$\OM$ be an open set in a Riemannian manifold $M$ and let $f\in C^{2}(\Omega)$, $f>0$ on $\Omega$. Then
\begin{equation}\label{eqThmP1}
 \lambda^{\ast}(\OM)\geq
 %\,\sup_{{\mathcal C}(\OM) } \left\{
 \inf_{ \OM} \left(-\frac{\Delta f}{f}\right).%\diver X-
%\vert X \vert^{2})\right\},
\end{equation}
\end{theorem}

%\begin{proof}
%Suppose that \eqref{eqThmP1} fails. Then, there exists a positive function $f \in \mathcal{C}(\Omega)$ and $\eps>0$ such that $-\Delta f \ge [\lambda^\ast(\OM)+\eps]f$ on $\Omega$. By a result in \cite{fischercolbrieschoen} and \cite{mosspie} (see also \cite{70}), the operator $-\Delta -[\lambda^\ast(\OM)+\eps]$ is non-negative on $\Omega$, that is,
%$$
%0 \le \lambda^\ast\big(-\Delta -[\lambda^\ast(\Omega)+\eps]\big) = \lambda^\ast(\Omega) - [\lambda^\ast(\Omega)+\eps] = -\eps <0,
%$$
%a contradiction.
%\end{proof}

We recall that, given a model $\mathbb{Q}^\kappa_g$ with $g>0$ on $(0,R_0)$, and given $R \in (0,R_0)$, the first eigenfunction $v$ of the geodesic ball $B_g(R)$ centered at $o$ is radial. This can be easily seen by proving that its spherical mean
$$
\bar{v}(r) = \frac{1}{g(r)^{\kappa -1}}\int_{\partial B_g(R)} v
$$
is still an eigenfunction associated to $\lambda_{1}(B_g(R))$ and using the fact that the space of first eigenfunctions has dimension $1$.  With a slight abuse of notation, we can thus identify the first eigenfunction $v \in C^\infty(B_g(R))$ of $B_g(R)$ with the solution $v\colon [0,R] \ra \erre$ of
\begin{equation}\label{primaautof}
\left\{ \begin{array}{l}
v'' + (\kappa -1)\displaystyle \frac{g'}{g}v' + \lambda_{1}(B_g(R))v = 0 \qquad \text{on } (0,R), \\[0.3cm]
v(0)=1, \quad v'(0)=0, \quad v(R)=0, \quad v>0 \text{ on } [0,R).
\end{array}\right.
\end{equation}
Note that, multiplying the ODE by $g^{\kappa-1}$, integrating and using the initial condition, one can easily argue that $v'<0$ on $(0,R]$.
\vspace{2mm}

We will need the following technical lemma, which  extends a result due to Bessa-Costa, see \cite[Lemma 2.4]{BS}.
\begin{lemma}\label{wolverine}
Let $\mathbb{Q}^\kappa_g$ be a model manifold with radial sectional curvature $-G(r)$, and suppose that $g'>0$ on $[0,R)$. Let $v \in C^2(B_g(R))$ be a first positive eigenfunction of $B_g(R) \subset \mathbb{Q}^\kappa_g$. If
\begin{equation}\label{rough}
\lambda_{1}(B_g(R)) \ge \kappa\|G_-\|_{L^\infty([0,R])}.
\end{equation}
Then  the following inequality holds:
\begin{equation}\label{wolverine2}
\kappa\frac{g'(t)}{g(t)} v'(t)+
\lambda_{1}(B_{%\mathbb{Q}_g^{m-n}
g}(R))v(t) \le 0, \qquad t \in(0,R].
\end{equation}
\end{lemma}
\begin{proof}
For simplicity of notation, we denote by $\lambda=\lambda_{1}(B_{%\mathbb{N}^n(\kappa)
g}(R))$. Multiplying \eqref{primaautof} by $g^{\kappa-1}$ we deduce that $v(t)$ satisfies the following differential equation:
\begin{equation}\label{eqSilvana2}
\left\{ \begin{array}{l}
\disp (g^{\kappa-1}v')' + \lambda g^{\kappa-1}v = 0 \qquad \text{on } (0,R), \\[0.2cm]
v(0)=1, \quad v'(0)=0, \quad v(R)=0, \quad v>0 \text{ on } [0,R).
\end{array}\right.
\end{equation}
Our aim is to deduce \eqref{wolverine2} via some modified Sturm-type arguments. In order to do so, we search for a positive function $\mu$ solving
\begin{align}
\label{eqmu}\kappa\mu'(t)\frac{g'(t)}{g(t)}+\lambda\mu(t)=0\,\,\textrm{ on} \,\,(0,R).
\end{align}
Integrating, we get that $\log \mu(t)=-\dfrac{\lambda}{\kappa}\displaystyle
\int_0^t\dfrac{g(s)}{g'(s)}\di s$, thus
\[ \mu(t)
=e^{\left(-\dfrac{\lambda}{\kappa}\displaystyle\int_0^t\dfrac{g(s)}{g'(s)}\di s\right)}.
\]
The above expression is well defined since $g'>0$ on $[0,R)$.

 Since
$
\mu'(t) = -\displaystyle\frac{\lambda}{\kappa}\dfrac{g(t)}{g'(t)}\mu(t)$
we deduce
\begin{align}\label{eqpacelli2}
%\begin{array}{ccl}
\mu'(t)v(t)\!-v'(t)\mu(t)\!&= -\dfrac{\lambda}{\kappa}\dfrac{g(t)}{g'(t)} e^{\!\left(-\dfrac{\lambda}{\kappa}\displaystyle\int_0^t\!\dfrac{g(s)}{g'(s)}\di s\right)}\!v(t)\!-\!v'(t)e^{\!\left(-\dfrac{\lambda}{\kappa}\displaystyle\int_0^t\!\dfrac{g(s)}{g'(s)} \di s\right)}\nonumber\\\nonumber
\\
=&\displaystyle-\frac{1}{\kappa}\frac{g(t)}{g'(t)}e^{\left(-\dfrac{\lambda}{\kappa}\displaystyle\int_0^t\dfrac{g(s)}{g'(s)}\di s\right)}
\left( \displaystyle \kappa\frac{g'(t)}{g(t)}v'(t) + \lambda
v(t)\right).
%\end{array}
\end{align}
 %e^{\left(-\dfrac{1}{\kappa}\displaystyle\int_0^t\frac{g}{g'}\right)}
From (\ref{eqpacelli2}) we see that $
\kappa\displaystyle\frac{g'(t)}{g(t)}v'(t) + \lambda v(t) \le 0
$ on $(0,R)$  if and only if
\[
\mu'(t)v(t)-v'(t)\mu(t) \ge 0 \qquad \text{on } (0,R),
\]
and we are going to prove this last inequality.

% By continuity, this also gives \eqref{eqmu} for $t=R$. ????? <---- What does it mean?

 Differentiating \eqref{eqmu} and multiplying by $(1/\kappa)$ both sides of the  equality, we have
\[\mu''(t)\frac{g'(t)}{g(t)}+\mu'(t)\left[G(t)-\left(\frac{g'(t)}{g(t)}\right)^2+\frac{\lambda}{\kappa}\right]=0,\]
that is,
\begin{align*}
\mu''(t)=-\mu'(t)\frac{g(t)}{g'(t)}\left[G(t)-\left(\frac{g'(t)}{g(t)}\right)^2+\frac{\lambda}{\kappa}\right].
\end{align*}
Since $\mu'(t)\dfrac{g(t)}{g'(t)}=-\dfrac{\lambda}{\kappa}\mu(t)\left(\dfrac{g(t)}{g'(t)}\right)^2$ we can rewrite $\mu''(t)$ in the following way:
\[\mu''(t)=\frac{\lambda}{\kappa}\mu(t)\left[G(t)\left(\frac{g(t)}{g'(t)}\right)^2-1+\frac{\lambda}{\kappa}\left(\frac{g(t)}{g'(t)}\right)^2\right].\]
Multiplying the above equation by $g^{\kappa-1}(t)$, and then adding and subtracting the term $(\kappa-1)g^{\kappa-2}(t)g'(t)\mu'(t)$, we obtain
\begin{equation}\label{eqSilvana4}
%\begin{array}{ccl}\displaystyle
(g^{\kappa-1}\mu')'(t) =-\lambda g^{\kappa-1}(t)\mu(t)
\left[-\frac{G(t)}{\kappa}\left(\frac{g(t)}{g'(t)}\right)^2\! -\frac{\lambda}{\kappa^2}\left(\frac{g(t)}{g'(t)}\right)^{2}\!+1\right].
%\end{a}
\end{equation}
Next, we multiply \eqref{eqSilvana4} by $v(t)$ and \eqref{eqSilvana2} by $-\mu(t)$, and we add them to get
\[
(g^{\kappa-1}\mu')'(t)v(t)-(g^{\kappa-1}v')'(t)\mu(t)=\frac{\lambda}{\kappa} g^{\kappa-1}(t)\mu(t) v(t)\left(\frac{g(t)}{g'(t)}\right)^2\left[G(t)+\frac{\lambda}{\kappa}
\right].
\]
Integrating from $0$ to $t$ gives
\begin{equation}
%\begin{array}{ccl}
g^{\kappa-1}\left(\mu'v-v'\mu \right)(t) =\displaystyle
\int_{0}^{t}\frac{\lambda}{\kappa} g^{\kappa-1}(s) \left(\frac{g(s)}{g'(s)}\right)^2\left[G(s)+\frac{\lambda}{\kappa}%+\frac{}{\kappa}
\right]
\mu(s)v(s) \di s.
%\end{array}
\end{equation}
Now, from \eqref{rough} we deduce that
\[
\displaystyle\frac{\lambda}{\kappa} g^{\kappa-1}(t) \left(\frac{g(t)}{g'(t)}\right)^2\left[G(t)+\frac{\lambda}{\kappa}%+\frac{}{\kappa}
\right]
\mu(t)v(t) \ge  0,\]
whence
$
\mu'(t)v(t)-v'(t)\mu(t) \ge 0$ for $t \in (0,R)$, as claimed.
\end{proof}
\begin{remark}\itshape{\label{rem_useful}
\emph{
It is important to find conditions to ensure \eqref{rough}. For instance, if  $-G(r) = B^2$, where $B$ is a positive constant,  then the solution $g_{_{B}}$ of \eqref{eqg} is
\begin{equation}\label{casosfera}
 g_{_{B}}(r)= B^{-1}\sin(Br), \qquad  \text{thus} \qquad  g_{_{B}}'>0 \quad \text{on } [0, \pi/(2B)).
\end{equation}The function  $g_{_{B}}$ yields the  model manifold $\mathbb{Q}_{g_{_{B}}}^{\kappa}=\mathbb{S}^{\kappa}(B^2)$, the $\kappa$-dimensional sphere of constant sectional curvature $B^{2}$ and diameter  ${\rm diam}_{\mathbb{S}^{\kappa}(B^2)}=\pi/B$. Note that the first eigenvalue of the geodesic ball of $\mathbb{S}^{\kappa}(B^2)$ of radius $R=\pi/2B$   is $\lambda_{1}(B_{\mathbb{S}^{\kappa}(B^2)}(\pi/2B))= \kappa B^2$ and   $v(r)=\cos(Br)$ is its first eigenfunction.}

\emph{ When   $-G(r) \le B^2$ and $R \le \pi/(2B)$, by Sturm's argument a solution  $g$ of \eqref{eqg} satisfies
$$
\frac{g'}{g} \ge  \frac{ g_{_{B}}'}{g_{_{B}}} >0 \quad \text{on } \left[0, \frac{\pi}{2B}\right).
$$
By  Cheng's Comparison Theorem  (version proved by Bessa-Montenegro in \cite{BeMo-cambridge}),}

\emph{$$\disp \lambda_{1}(B_g(R)) \ge \disp \lambda_{1}(B_{g_{_B}}(R)),\quad R\in \left[0, \frac{\pi}{2B}\right) .$$ }

\emph{\noindent  In order to get  $\disp \lambda_{1}(B_g(R)) \ge\kappa \Vert G_-\Vert_{L^\infty([0,R))}$ it is sufficient to have \begin{equation}\lambda_{1}(B_{g_{_B}}(R))=\lambda_{1}(B_{\mathbb{S}^{\kappa}(B^{2})}(R))\geq \kappa \Vert G_{-}\Vert_{L^\infty([0,R))}.\label{eqLambda}\end{equation}
On the other hand, we can see $  \kappa \Vert G_{-}\Vert_{L^\infty([0,R))}$ as a first eigenvalue of a ball of radius  $\tilde{R}$ in a $\kappa$-dimensional sphere of sectional curvature $\tilde{B}^{2}$\!, i.e.
$$ \kappa \Vert G_{-}\Vert_{L^\infty([0,R))} =\lambda_{1}(B_{\mathbb{S}^{\kappa}(\tilde{B}^{2})}(\tilde{R})),$$ where  $\tilde{R}=\frac{\pi }{2 \sqrt{\Vert G_{-}\Vert_{L^{\infty}([0, R))}}}$ and  $\tilde{B}^{2}=\Vert G_{-}\Vert_{L^{\infty}([0, R))}$. }
\vspace{1mm}

\noindent \emph{We conclude that the inequality \eqref{eqLambda} holds, thus $\disp \lambda_{1}(B_g(R)) \ge\kappa \Vert G_-\Vert_{L^\infty([0,R))}$, whenever   $$
R \le \frac{\pi}{2 \sqrt{\|G_-\|_{L^\infty([0, R))}}}\cdot
$$}}

\end{remark}
\begin{remark}
\emph{We remark that if
\[
t\int_t^\infty G_{-}(s)\di s\leq\frac{1}{4} \qquad \text{for every } t\in \erre^+,
\]
where $G_{-}(s)=\max\left\{0, -G(s)\right\}$, both $g$ and $g'$ are strictly positive on $\erre^+$. This criterion has been proved in  \cite[Prop. 1.21]{bmr2}.
}
\end{remark}
\noindent \textbf{A preliminary computation.}\\
From now on, we will consider the case when the ambient space is a warped product $W^{n+q}=N\times_{f}Q$
of two Riemannian manifolds $(N^n,\left\langle ,\right\rangle_N)$ and $(Q^q,\left\langle ,\right\rangle_Q)$, %are Riemannian manifolds and $f:N \ra \erre_{+}$ is a smooth function.The product manifold $N\times Q$ is
 with
the Riemannian metric on $W$  given by
$$
\left\langle \!\left\langle \,,\,\right\rangle\!\right\rangle=\pair{\,,\,}_N+f^2\pair{\,,\,}_Q
$$
for some smooth positive function $f\colon N \ra \erre^{+}$. We fix the index convention
\[1\leq j,k\leq n, \quad n+1\leq\alpha,\beta \leq n+q.\]
For $(p,q)\in W$, we choose a chart $(U,\psi)$ on $N$ around $p$, with coordinate tangent basis $\left\{\partial_j\right\}=\left\{\partial/\partial \psi_j\right\}$, and a chart $(V,\phi)$ on $Q$ around $q$, with basis $\left\{\partial_\alpha\right\}=\left\{\partial/\partial \phi_\alpha\right\}$. Then, with respect to the product chart $(U\times V,\psi\times\phi)$ around $(p,q)$, the Hessian of $F$ at $(p,q)$ has components
{\small\begin{equation}\label{hessF}
\left\{ \begin{array}{lcl}
\Hess_{W}
F(\partial_j,\partial_\kappa) & = & \disp \Hess_{N}
F(\partial_j,\partial_\kappa),\\[0.2cm]
\Hess_{W} F(\partial_j,\partial_\alpha) & = & \disp \partial_j\partial_\alpha F-\frac{1}{f}\partial_j f\partial_\alpha F,\\[0.2cm]
\Hess_{W} F(\partial_\alpha,\partial_\beta) & = & \disp \Hess_Q
F(\partial_\alpha,\partial_\beta)+\frac{1}{f}\left\langle \nabla^N f, \nabla F\right\rangle_N\left\langle \!\left\langle \partial_\alpha, \partial_\beta\right\rangle\!\right\rangle,
\end{array}
\right.
\end{equation}}
\noindent where {\small $\Hess_N \!F$} and {\small $\Hess_Q \!F$} mean  respectively {\small$\Hess(F\circ i_N)$} and {\small$\Hess(F\circ i_Q)$} and the inclusions are given by
$$
\begin{array}{ll}
i_N \colon (N,\left\langle ,\right\rangle_N)\rightarrow N\times_f\left\{q\right\}\subseteq N\times_fQ, & \quad x \mapsto (x,q), \\[0.1cm]
i_Q \colon (Q,\left\langle ,\right\rangle_Q)\rightarrow \left\{p\right\}\times_f Q\subseteq N\times_fQ, & \quad y \mapsto (p,y).
\end{array}
$$
From \eqref{hessF} we observe that if $F(p,q)=f(p)\cdot h(q)$, where $f$ is the warping function and $h:Q\rightarrow\mathbb{R}$ is a smooth function on $Q$, then $\Hess_{W} F$ has a block structure, that is
\[\Hess_{W} F(X,Z)=0 \quad \forall X\in T_{(p,q)}\big(N\times_f\left\{q\right\}\big), \,Z\in T_{(p,q)}\big(\left\{p\right\}\times_f Q\big).\]
More precisely, we have the %we obtain the
following result.
\begin{lemma}\label{prop1}
Let $F \in C^\infty(N\times_f Q)$ be given by $F(p,q)=f(p)\cdot h(q)$, where $h \in C^\infty(Q)$. %:%Q\rightarrow\mathbb{R}$. %be a minimal $m$-dimensional submanifold,
%where $Q$ has radial sectional curvature $K(\gamma (t))(\gamma'(t), v)\leq \kappa$, $v\in T_{\gamma (t)}Q$, $\vert v\vert =1$, $v\perp \partial t$,  along the geodesics $\gamma (t)$ issuing from  a point  $x_{0}\in Q$. Let $\OM \subset \varphi^{-1}(N\times_f B_{Q}(x_{0},r))$ be a connected component, where $r < \min \{\inj(x_{0}), \pi/2
%\sqrt{\kappa} \}$, $(\pi/2
%\sqrt{\kappa}=\infty$ if $\kappa\leq 0)$.
Then
\begin{align}\label{hesswp}
\left\{\begin{array}{lll}
\Hess_{W}%_{N\times_f Q}
F(X,Y)
=h\Hess_N
f(X,Y),% \nonumber%\quad\forall X,Y\in T_{(p,q)}N\times\{q\}.
\\[0.2cm]
\Hess_{W}%_{N\times_f Q}
F(X,Z)=0, \\
\Hess_{W}%_{N\times_f Q}
F(Z,W)=f\Hess_Q
h(Z,W)+h\dfrac{\left|\nabla^N f\right|_N^2}{f}\left\langle \!\left\langle Z,W\right\rangle\!\right\rangle,%\;\forall Z,W\!\in T_{(p,q)}\{p\}\!\times\! Q.
\end{array} \right.
 \end{align}
for every $X,Y\in T_{(p,q)}\big(N\times_f\{q\}\big)$ and $Z,W\in T_{(p,q)}\big(\{p\}\times_f Q\big)$.%If  $\Omega$ is bounded then inequality \eqref{estwp} is strict. Here $\mathbb{Q}^{m-1}(\kappa)$ is the  $(m-1)$-dimensional simply connected space form of constant sectional curvature $\kappa$.
  \end{lemma}
%\begin{proof}
%Let $(U,\psi)$ be a coordinates system on $N$ around $p$ and denote with $\partial_i=\partial/\partial\psi_i$ for every $i=1,\ldots,n$ and $(V,\phi)$ be a coordinate system $Q$ around $q$ and we will denote by $\partial_\alpha=\partial/\partial\phi_\alpha$ for every $\alpha=1,\ldots,q$. % and %$\left\{\delta_\alpha\right\}_{\alpha=1}^{q-1}$ coordinates frame on $Q$
%We consider on $N\times_f Q$ a coordinate system around $(p,q)$ formed by coordinates frames of $N$ and $Q$, such that
%Note that for any smooth function $\widetilde{h}:N\times_f Q\rightarrow\mathbb{R}$ we have
%where the indexes $1\leq i,j\leq n$ and $n+1\leq\alpha,\beta\leq n+q$.
% \eqref{hesswp}.
%\end{proof}
\section{Main results}
Let $\varphi\colon M^m\rightarrow N^n\times_f Q^q$, $m>n$, be a minimal immersion. Hereafter, we shall require the following
\begin{assunz}\label{assumpt}\emph{Define $\rho_{_Q}(x)={\rm dist}_{Q}(o,x)$ and suppose that the radial sectional curvature of $Q$ satisfies
\[K^\rad_{Q}\leq-G(\rho_{_Q}),\,\,\,\textrm{where}\,\, G\in C^{\infty}(\mathbb{R}_{0}^{+}).
\]
We assume that the solution $g$ of \eqref{eqg} is positive and $g'>0$ on $[0,R)$, and that $B_Q(o,R)\subseteq Q\backslash \mathrm{cut}(o)$.}
\end{assunz}
Let $v:\overline{B_{g}(R) }\rightarrow \mathbb{R}$ be the first eigenfunction of the ball $B_g(R) \subset \mathbb{Q}^{m-n}_g$. As remarked, $v>0$ on $B_g(R)$, $v$ is  radial and (up to normalization) solves
\begin{equation}\label{sil01}\left\{\begin{array}{ll}
v''(t) + (m-n-1)\dfrac{g'(t)}{g(t)} v'(t) +
\lambda_{1}(B_g(R)) v(t) = 0, \;\;\; t\in
(0,R)\\[0.3cm]
v(0)=1, \quad v(R)=0, \quad v>0\textrm{ on }[0,R), \quad v'<0 \text{ on }(0,R].
\end{array}\right.\\
\end{equation}
Observe that, when $m=n+1$, the equation simply becomes $$v''(t) + \lambda_{1}(B_g(R))v(t)=0.$$

 %Choose the first eigenfunction that satisfies the initial conditions
 %$v(0) = 1$ and $v'(0)=0$.

  \begin{theorem}\label{thm2}
  Let $\varphi : \!M^m\rightarrow N^n\times_f Q^q $ be an $m$-dimensional submanifold minimally immersed into $N^n\times_f Q^q$, where $Q$ satisfies Assumption \ref{assumpt} and $m>n$. Suppose that
  the warping function $f$ satisfies
  \begin{eqnarray}\label{condf}
  \Hess_N f(\cdot,\cdot)-\frac{\left|\nabla^N f\right|_N^2}{f}\left\langle ,\right\rangle_N\leq0.\end{eqnarray}
%   and $Q$ has radial sectional curvature $K(\gamma (t))(\gamma'(t), v)\leq \kappa$, $v\in T_{\gamma (t)}Q$, $\vert v\vert =1$, $v\perp \partial t$,  along the geodesics $\gamma (t)$ issuing from  a point  $x_{0}\in Q$.
Let $U\subseteq N$ be an open subset, and let
$\OM \subset \varphi^{-1}(U\times_f B_Q(o,R))$ be a connected component. Then, if $R$ is such that
\begin{equation}\label{buono}
R \le \frac{ \pi}{2 \sqrt{\|G_-\|_{L^\infty([0,R))}}}
\end{equation}
the following estimate holds:
\begin{equation}\label{estwp}
\lambda^*(\Omega)\! \geq \! \inf_{p \in U} \left(\frac{\lambda_{1}(B_g(R)) - m\,\vert \nabla^N f\vert_N^2(p)}{\vert f(p)\vert^2}\right),
\end{equation}
where $B_g(R)$ is the geodesic ball of radius $R$ in the model manifold $\mathbb{Q}_g^{m-n}$ or the interval $[-R,R]$ if $m=n+1$.
%If  $\Omega$ is bounded then inequality \eqref{estwp} is strict. %Here $\mathbb{Q}^{m-n}(\kappa)$ is the  $(m-n)$-dimensional simply connected space form of constant sectional curvature $\kappa$.
  \end{theorem}
\begin{proof}We start defining
 $F\colon U\times_f B_{Q}(o, R)\rightarrow \mathbb{R}$ by $F(p,q)= f(p) \cdot h(q)$,  where  $h \in C^\infty(B_{Q}(o,R))$ is given by $h(q)=(v\circ \rho_{_Q})(q)$ and $v \in C^\infty([0,R])$ is the solution of \eqref{sil01}. By Theorem \ref{thmBM1}, we have  that \
\begin{equation}\label{basebarta}
\lambda^{\ast}(\OM)\geq \inf_{\OM}\left(-\displaystyle\frac{\Delta (F\circ\varphi)}{F\circ\varphi}\right).
\end{equation}
We are going to give a lower bound for $-\Delta (F\circ\varphi)/(F\circ\varphi)$.  Let $x\in \OM$ and let
 $ \{e_{1},\ldots, e_{m}\}$ be an orthonormal basis
for $T_{x}\OM$. %is given by %(recall
 Let $\varphi (x)=(p(x),q(x))$, $t(x)=\rho_{_Q}(q(x))$  and denote by $P_{_N}:T_{_{(p,q)}}\big(N\times_f Q\big)\rightarrow T_{_{(p,q)}}\big(N\times_f \left\{q\right\}\big)$ and $P_{_Q}:T_{_{(p,q)}}\big(N\times_f Q\big)\rightarrow T_{_{(p,q)}}\big(\left\{p\right\}\times_fQ \big)$ the orthogonal projections onto the tangent spaces of the two fibers. Then, by  \eqref{eqBF3} and the minimality of $M$, the Laplacian of $F\circ\varphi$ at $x$ has the expression
{\small \begin{align*}
\Delta\,(F\circ\varphi )(x)&=\!
\sum_{i=1}^{m}\Hess_{W}%_{(N\times_fQ)}
F (\varphi
(x))(e_i,e_i)\nonumber\\&=\sum_{i=1}^{m}\Big[\Hess_{W}%_{(N\times_fQ)}
F (\varphi
(x))(P_{_N}\!e_i,P_{_N}\!e_i)+\Hess_{W}%_{(N\times_fQ)}
F (\varphi
(x))(P_{_Q}\!e_i,P_{_Q}\!e_i)\Big]%\\&
\end{align*}}where $W=N\times_f Q$.
Using Lemma \ref{prop1}, we deduce and writing  $t=t(x)$ for simplicity of notation,
{\small \begin{align}\label{eqLaplacian}
\Delta\,(F\circ\varphi )(x)&=v(t)\sum_{i=1}^{m}\Hess_{N}
f(P_{_N}\!e_{i},P_{_N}\!e_{i})(p)+f(p)\sum_{i=1}^{m}\Hess_{Q}
v(t)(P_{_Q}\!e_{i}, P_{_Q}\! e_{i})\nonumber\\&\quad+v(t)\frac{\left|\nabla^Nf\right|_N^2}{f}(p)\sum_{i=1}^{m}\left\langle\!\left\langle P_{_Q}\!  e_i,P_{_Q}\! e_i\right\rangle\!\right\rangle.
\end{align}}
%where .
Let $\{E_1, \ldots,E_n\}$ be an orthonormal basis for $T_{p}N$, and consider the tangent  basis $\big\{ \partial/\partial \rho_{_Q},\left\{ \partial/\partial\theta^{\gamma}\right\}_{\gamma=n+2}^{n+q}\big\}$,    associated to normal coordinates at $Q$. Then the set $\left\{\xi_{l}\right\}_{l=1}^{n+q}$ %for $T_{_(p,q)}\big(N\times_f Q\big)$
given by\[\xi_{j}=E_j\;\;\forall j=1,\ldots,n,\;\;\;\xi_{n+1}=\frac{1}{f}\frac{\partial}{\partial \rho_{_Q}}, \quad\xi_{\gamma}=\frac{1}{f}\frac{\partial}{\partial\theta^{\gamma}}\;\;\forall\gamma=n+2,\ldots,n+q\]
%and $\{ \partial/\partial \rho_{_Q}, \partial/\partial\theta^{2}, \ldots, \partial/\partial\theta^{q}\}$ is an orthonormal basis for $T_{q}Q$ (polar coordinates).
%Let  $\{e_{1}, \ldots, e_{m}\}$ be an orthonormal basis for $T_{x}\OM$ and
is an orthonormal basis of $T_{(p,q)}\big(N\times_fQ\big)$. So, we can write $e_i$ as a linear combination of vectors of this basis in the following way:
\[e_{i}=\sum_{j=1}^{n}a_{i}^j\cdot \xi_j+b_{i}\cdot\xi_{n+1} +\sum_{\gamma=n+2}^{n+q}c_{i}^{\gamma}\cdot\xi_{\gamma},\]
for constants $a_{i}^j, b_{i}, c_{i}^{\gamma}$ satisfying
\begin{equation}\label{eqcons}
\sum_{j=1}^{n} (a_{i}^j)^{2}+b_{i}^{2}+\sum_{\gamma=n+2}^{n+q}(c_{i}^{\gamma})^{2}=1, \,\,\forall i=1,\ldots,m .
\end{equation}
%Computing $\Delta (F\circ\varphi)(x)$ we have%, (recall that $\varphi (x)=(p,q)$ and we are letting $t=\rho_{Q}(q)$)
%
From
$$
\nabla^Qv(t) = v'(t) \frac{\partial}{\partial \rho_{_Q}}, \quad \Hess_Q v(t) = v'(t) \Hess_Q \rho_{_Q} + v''(t) \di \rho_{_Q} \otimes \di \rho_{_Q},
$$
we can rewrite \eqref{eqLaplacian} in the following way:
{\small \begin{align*}
\Delta (F\circ\varphi )(x)=&v(t)\sum_{i=1}^{m}\Hess_N f(P_{_N}\! e_i,P_{_N}\!e_i)(p)+f(p) \sum_{i=1}^{m}\Big[P_{_Q}\!e_{i}(v'(t))\left\langle  \frac{\partial}{\partial\rho_{_Q}} , P_{_Q}\!e_{i}\right\rangle_Q \nonumber\\&+ v'(t)\Hess_{_Q}
\rho_{_Q}(P_{_Q}\!e_{i},P_{_Q}\!e_{i})\Big]+v(t)\frac{\left|\nabla^N f\right|_N^2}{f}(p)\sum_{i=1}^{m}\left\langle\!\left\langle  P_{_Q} \!e_i,P_{_Q} \!e_i\right\rangle\!\right\rangle\nonumber \\
=&v(t)\sum_{i=1}^{m}\left(\Hess_{N} f(P_{_N} \!e_i,P_{_N}\!e_i)+\frac{\left|\nabla^N f\right|_N^2}{f}\Big(1-\left\langle\!\left\langle  P_{_N}\! e_i,P_{_N}\!e_i\right\rangle\!\right\rangle%b_i^2+\sum_{\gamma=n+2}^{n+q}(c_i^{\gamma})^2
\Big)\right)(p)\\&+\frac{1}{f(p)}\left(\!v''(t)\sum_{i=1}^{m}b_{i}^{2}+v'(t)\sum_{i=1}^{m}\sum_{\gamma=n+2}^{n+q}(c_{i}^{\gamma})^{2}\Hess_{_Q} \rho_{_Q}\left(\frac{\partial}{\partial\theta^{\gamma}},\frac{\partial}{\partial\theta^{\gamma}}\right)\!\right).%\nonumber\\&+v(t),\nonumber
\end{align*}}
Using \eqref{condf} and the fact that $v$ is positive we have
\begin{eqnarray}
-\Delta (F\circ\varphi )(x)&\geq & -\,\frac{1}{f(p)}\Big[mv(t)\vert \nabla^N f\vert_N^2(p) +v''(t)\sum_{i=1}^{m}b_{i}^{2} \nonumber \\ &&+\,v'(t)\sum_{i=1}^{m}\sum_{\gamma=n+2}^{n+q}(c_{i}^{\gamma})^{2}
\Hess_{_Q} \rho_{_Q}\left(\frac{\partial}{\partial\theta^{\gamma}},\frac{\partial}{\partial\theta^{\gamma}}\right)\Big].\nonumber
\end{eqnarray}
Since $v'(t)\leq 0$, we can apply the Hessian Comparison Theorem, to obtain
{\small\begin{eqnarray*}
-\Delta (F\circ\varphi )(x)&\geq & -\frac{1}{f(p)}\Big[mv(t)\left|\nabla^N f\right|_N^2(p)+v''(t)\sum_{i=1}^{m}b_{i}^{2}\\ && +
v'(t)\frac{g'(t)}{g(t)}\sum_{i=1}^{m}\sum_{\gamma=n+2}^{n+q}(c_{i}^{\gamma})^{2}\Big]\nonumber\\
&=&-\frac{1}{f(p)}\Big[v''(t)\sum_{i=1}^{m}b_{i}^{2}+
v'(t)\frac{g'(t)}{g(t)}\Big(m-\sum_{i=1}^{m}\sum_{j=1}^{n}(a_{i}^j)^{2}-\sum_{i=1}^{m}b_{i}^{2}\Big)\\ && +\,mv(t)\left|\nabla^N  f\right|_N^2(p)\Big]
\end{eqnarray*}}
where the last equality follows by an algebraic manipulation that uses \eqref{eqcons} summed for $i=1,\ldots,m$. Now, by a simple rearranging,
{\small\begin{align*}
-\Delta (F\circ\varphi )(x)&\geq-\frac{1}{f(p)}\Big[v''(t)+(m-n-1)v'(t)\frac{g'(t)}{g(t)}-v''(t)\left(1\!-\!\sum_{i=1}^{m}b_{i}^{2}\right)
\nonumber\\&\qquad+\!v'(t)\frac{g'(t)}{g(t)}\left(\!n\!-\!\sum_{i=1}^{m}\sum_{j=1}^{n}(a_{i}^j)^{2}\!+\!1-\!\sum_{i=1}^{m}b_{i}^{2}\right)\!+\!mv(t)\left|\nabla^N f\right|_N^2(p)\Big].
\end{align*}}
From \eqref{sil01} we get
{\small\begin{align}\label{eqSilvana9}
-\Delta (F\circ\varphi) (x)\!&\geq\frac{v(t)}{f(p)}\left(\lambda_{1} (B_{%\mathbb{Q}^{m-1}(\kappa)
g}(R))-\!m\left|\nabla^N f\right|_N^2(p)\right)\nonumber\\&\,+\!\frac{1}{f(p)}\left[v''(t)\!\left(1\!-\!\sum_{i=1}^{m}b_{i}^{2}\right)
\!-\!v'(t)\frac{g'(t)}{g(t)}\!\left(\!n\!-\!\sum_{i=1}^{m}\sum_{j=1}^{n}(a_{i}^j)^{2}\!+\!1-\!\sum_{i=1}^{m}b_{i}^{2}\right)\right].
\end{align}}
\noindent We claim that the last line of (\ref{eqSilvana9}) is nonnegative, that is,
\begin{equation} \label{eqSilvana10}
v''(t)\left(1-\sum_{i=1}^{m}b_{i}^{2}\right)
-v'(t)\displaystyle\frac{g'(t)}{g(t)}\left(n-\sum_{i=1}^{m}\sum_{j=1}^{n}(a_{i}^j)^{2}+1-\sum_{i=1}^{m}b_{i}^{2}\right)\geq 0.
\end{equation}
To prove this, we substitute $v''(t)=-(m-n-1)v'(t)\displaystyle\frac{g'(t)}{g(t)}-\lambda_{1} (B_{%\mathbb{Q}^{m-n}(\kappa)
g}(R))v(t)$ in  (\ref{eqSilvana10}) to get
 \begin{eqnarray}
v''(t)\left(1-\sum_{i=1}^{m}b_{i}^{2}\right)
-v'(t)\displaystyle\frac{g'(t)}{g(t)}\left(n-\sum_{i=1}^{m}\sum_{j=1}^{n}(a_{i}^j)^{2}+1-\sum_{i=1}^{m}b_{i}^{2}\right)&=& \nonumber \\
&&\nonumber \\
-\left((m-n)v'(t)\displaystyle\frac{g'(t)}{g(t)}+
\lambda_{1} (B_{%\mathbb{Q}^{m-n}(\kappa)
g}(R))v(t)\right)\left(1-\sum_{i=1}^{m}b_{i}^{2}\right)&& \\
&&\nonumber \\
-v'(t)\displaystyle\frac{g'(t)}{g(t)}\left(n-\sum_{i=1}^{m}\sum_{j=1}^{n}(a_{i}^j)^{2}\right), & & \nonumber
 \end{eqnarray}
so that \eqref{eqSilvana10} is equivalent to show that
\begin{eqnarray}\label{eqIgualdade}
&&\nonumber \\
-\left((m-n)v'(t)\displaystyle\frac{g'(t)}{g(t)}+
\lambda_{1} (B_{%\mathbb{Q}^{m-n}(\kappa)
g}(R))v(t)\right)\left(1-\sum_{i=1}^{m}b_{i}^{2}\right)&& \\
&&\nonumber \\
-v'(t)\displaystyle\frac{g'(t)}{g(t)}\left(n-\sum_{i=1}^{m}\sum_{j=1}^{n}(a_{i}^j)^{2}\right)&\geq &0.\nonumber
 \end{eqnarray}
Now, in our assumption \eqref{buono}, by Remark \ref{rem_useful} it holds
$$
\lambda_{1}(B_g(R)) \ge (m-n)\|G_-\|_{L^\infty([0,R])}.
$$
Hence, applying Lemma \ref{wolverine} we infer that
$$
(m-n)v'(t)\displaystyle\frac{g'(t)}{g(t)}+
\lambda_{1} (B_{%\mathbb{Q}^{m-n}(\kappa)
g}(R))v(t)\le 0.
$$
Moreover, it is clear that $\left(1-\sum_{i=1}^{m}b_{i}^{2}\right)\geq 0$, and  finally we observe the inequality \[\sum_{i=1}^{m}\sum_{j=1}^{n}(a_{i}^j)^{2}=\sum_{j=1}^{n}\left(\sum_{i=1}^{m}\left\langle\! \left\langle e_i,\xi_j\right\rangle\!\right\rangle\right)=\sum_{j=1}^{n}\left|P_{_M}\xi_j\right|^2\leq\sum_{j=1}^{n}\left|\xi_j\right|^2=\sum_{j=1}^{n}1=n,\]
where $P_{_M}$ is the projection on $M$.%

Keeping in mind that $v'\le 0$, this concludes the proof of the claimed \eqref{eqIgualdade}. From (\ref{eqSilvana9}) we have
\begin{equation}\label{bgr}
\displaystyle -\frac{\Delta (F\circ\varphi) }{F\circ\varphi}(x)\geq\frac{1}{f^2(p)} \left(\lambda_{1} (B_{%\mathbb{Q}^{m-n}(\kappa)
g}(R))-m\left|\nabla^N f\right|_N^2(p)\right).
\end{equation}
Therefore, by \eqref{basebarta} we conclude the desired \eqref{estwp}.
\end{proof}
\begin{remark}
\emph{In the case $N=\mathbb{R}$, we observe that the mean curvature function of the fibers $\left\{p\right\}\times_fQ$ is given by $\mathcal{H}(y)={f'(y)}/{f(y)}$. Therefore, condition \eqref{condf} is equivalent to $f\mathcal{H}'\leq0$, that is, $\mathcal{H}'\leq0$. There exists a large class of functions for which $\mathcal{H}' \le 0$. For instance, $f(y) = \mathrm{constant}$, $f(y)=y$ and $f(y)= e^{cy}$, where $c \in \erre$.
}
\end{remark}
\section{Applications}
To show the generality of Theorem \ref{thm2}, we conclude this paper with a number of different examples, and we discuss the sharpness of the estimates produced.\vspace{-2mm}

\subsection{Cylinders}

Considering $f=1$ and $N=\erre$ in Theorem \ref{thm2} we obtain a generalized version of Theorem 1.1 of  \cite{BS}.
  \begin{corollary}\label{thm2}Let $\varphi \colon \!M^m\rightarrow \mathbb{R}\times Q^q$ be an  $m$-dimensional   submanifold minimally immersed into $\mathbb{R}\times Q^q$. Suppose that  $Q$ satisfies the Assumption \ref{assumpt}.
Let %$U\subseteq \mathbb{R}$ be an open subset, and let
  $\OM \subset \varphi^{-1}(\mathbb{R}\times B_Q(o,R))$ be a connected component with $$R\leq \frac{\pi }{2\sqrt{\Vert G_{-}\Vert_{L^{\infty}([0,R))}}} \cdot$$
 Then
$$
\lambda^*(\Omega) \geq \lambda_{1}(B_{%\mathbb{Q}^{m-1}(\kappa)
g}(R)).
 $$
 Here $B_g(R)$ is a geodesic ball  of radius $R$ in an $(m-1)$-dimensional model manifold $\mathbb{Q}_g^{m-1}$.
%If  $\Omega$ is bounded then inequality (\ref{sil-1}) is strict. %Here $\mathbb{Q}^{m-1}(\kappa)$ is the $(m-1)$-dimensional space form.
  \end{corollary}
%
%
%
%HEREAFTER, I HAVE NOT CHECKED THE PAPER. HOWEVER, COROLLARIES ARE INTERESTING. PLEASE ADD A COROLLARY REGARDING IMMERSIONS INTO REGIONS OF THE UNIT SPHERE!!! NOTE THAT $\|e^{2t}\|_{L^\infty(\erre)}$ DOES NOT MAKE SENSE!! IT SHALL BE RESTRICTED TO AN OPEN, BOUNDED INTERVAL $U = (a,b)$, FOR WHICH THE $L^\infty$-NORM IS IMMEDIATE TO COMPUTE.\\
%\par

In particular, when $Q^q=\mathbb{R}^q$ in the last corollary we get the following result in the Euclidean space proved by Bessa and Costa in \cite{BS}.
\begin{corollary}\label{euclidean}Let $\varphi\colon M^m\rightarrow \,\mathbb{R}^{q+1}$ be an $m$-dimensional submanifold minimally immersed into $\mathbb{R}^{q+1}.$
Let $\OM \subset \varphi^{-1}(\erre \times B_{\mathbb{R}^q}(o,R))$ be a connected component.
Then \begin{equation}\label{hyp}
\lambda^*(\Omega) \geq %
%\frac{
\lambda_1(B_{\mathbb{R}^{m-1}}(o,R))%}{\left\|\mathrm{e}^{2y}\right\|_{L^\infty(a,b)}}
=\left(\frac{c_{m-1}}{R}\right)^2.
%{\left\|\cosh t\right\|^2_{L^\infty(U)}}-m\left\|\tanh t\right\|^2_{L^\infty(U)}.
 \end{equation}
  %If  $\Omega$ is bounded then inequality (\ref{sil-1}) is strict.
Here %$c={c_0}/{\left\|\mathrm{e}^{2y}\right\|_{L^\infty(a,b)}}$ where
$c_{m-1}$ is the first zero of the $J_{(m-1)/2-1}$-Bessel function.
 \end{corollary}

\subsection{Pseudo-hyperbolic and hyperbolic spaces}

The pseudo-hyperbolic spaces, introduced by Tashiro in \cite{tas}, are  warped products $\mathbb{R}\times_f Q^q$ with
$$
(i) \ \  f(y) = a e^{b y}, \qquad \text{or} \qquad (ii) \ \ f(y) = a \cosh(b y),
$$
for some constants $a,b>0$. In the case $(i)$, we observe that condition \eqref{condf} is satisfied, as it shows
$$
f'' - \frac{(f')^2}{f} = \left\{ \begin{array}{ll} 0 & \quad \text{in case } (i), \\[0.1cm]
a b^2/\cosh(b y) > 0 & \quad \text{in case } (ii).
\end{array}\right.
$$
%satisfies
%$%f''(t)-\frac{(f'(t))^2}{f(t)}=0
%\mathcal{H}'=0\quad\textrm{for every } t\in\mathbb{R}$ %and the function $f(t)=\cosh t$ satisfies %\[%f''(t)-\frac{(f'(t))^2}{f(t)}=
%\mathcal{H}'=\mathrm{sech }\,t<0 \quad\textrm{for every } t\in(0,+\infty),\]
%we obtain the following result.
We state the following corollary in the  case $f(y)=e^{by}$.
\begin{corollary}\label{corohyperbolic}
Let $\varphi:M^m\rightarrow \mathbb{R}\times_{e^{by}} Q^q$ be  an $m$-dimensional submanifold minimally immersed into $\mathbb{R}\times_{e^{by}} Q^q $. Suppose that $Q$ satisfies Assumption \ref{assumpt}. Let $\OM \subset \varphi^{-1}\big((\alpha,\beta)\times_{e^{by}} B_Q(o,R)\big)$ be a connected component with $$R\leq \frac{\pi }{2\sqrt{\Vert G_{-}\Vert_{L^{\infty}([0,R))}}}\cdot $$
Then,
%, where $Q$ has radial sectional curvature $K(\gamma (t))(\gamma'(t), v)\leq \kappa$, $v\in T_{\gamma (t)}Q$, $\vert v\vert =1$, $v\perp \partial t$,  along the geodesics $\gamma (t)$ issuing from  a point  $x_{0}\in Q$. Let $\OM \subset \varphi^{-1}(\mathbb{R}\times_f B_{Q}(x_{0},r))$ be a connected component, where $r < \min \{\inj(x_{0}), \pi/2
%\sqrt{\kappa} \}$, $(\pi/2
%\sqrt{\kappa}=\infty$ if $\kappa\leq 0)$.
\begin{equation}\label{sil-1}
\lambda^*(\Omega) \geq \frac{\lambda_{1}(B_{%\mathbb{Q}^{m-1}(\kappa)
g}(R))}{e^{2b\beta}}-mb^2.
\end{equation}
Here $B_g(R)$ is the geodesic ball of $(m-1)$-dimensional model space $\mathbb{Q}_g^{m-1}$.
\end{corollary}
Foliating through horospheres, we can represent the hyperbolic space $\mathbb{H}^{q+1}$ as the warped product $ \erre\times_{e^y} \mathbb{R}^q$. %So that, we show
By Corollary \ref{corohyperbolic} we have  the following eigenvalue estimate.
%Likewise the $n$-dimensional Euclidean sphere $\mathbb{S}^{n+1}$ %can be realized as the following warped product $(-\pi,/2,\pi/2)\times_{\cos y}\mathbb{S}^{n}$ whose warping function
%satisfies \eqref{condf} on $(-\pi,/2,\pi/2)$.
%In the hyperbolic case, we get the following result.

\begin{corollary}\label{hyperbolic}
Let $\varphi \colon\! M^m\rightarrow \mathbb{H}^{q+1}$ be an $m$-dimensional submanifold minimally immersed into $\mathbb{H}^{q+1}$.
 Let $\OM \subset \varphi^{-1}((-\infty, \beta)\times_{e^y}B_{\mathbb{R}^q}(o,R))$ be a connected component.
Then \begin{equation}\label{cylEucl}
\lambda^*(\Omega) \geq %
%\frac{
\frac{\lambda_1(B_{\mathbb{R}^{m-1}}(o,R))}{e^{2\beta}}%}{\left\|\mathrm{e}^{2y}\right\|_{L^\infty(a,b)}}
-m=e^{-2\beta}\left(\frac{c_{m-1}}{R}\right)^2-m,
%{\left\|\cosh t\right\|^2_{L^\infty(U)}}-m\left\|\tanh t\right\|^2_{L^\infty(U)}.
 \end{equation}
where %$c={c_0}/{\left\|\mathrm{e}^{2y}\right\|_{L^\infty(a,b)}}$ where
$c_{m-1}$ is the first zero of the $J_{(m-1)/2-1}$-Bessel function.
\end{corollary}
\subsection{Cones}
A $(q+1)$-dimensional cone $\mathcal{C}^{q+1}(Q)\subseteq \erre^m$ over an open subset $Q\subset\esse^q$ can be seen as the warped product %note that
$\mathcal{C}^{q+1}(Q)= (0,+\infty)\times_f Q $ where $f(y)=y$. In order to match with Assumption \ref{assumpt}  we shall suppose that   $Q \subset B_{\mathbb{S}^q}(o,R)$ for some $R \le \pi/2$. More generally, we can consider cones $\mathcal{C}^{q+1}(Q)$ over open subsets $Q\subset W$ of Riemannian manifolds $W$ with $Q$ satisfying Assumption \ref{assumpt}.  We have the following result. %implies
%\[\left\langle \varphi(x), v\right\rangle\geq a|\varphi(x)| > 0 \quad\forall x\in M\]
\begin{corollary}\label{prin}
Let $\varphi \colon M^{m}\rightarrow \mathcal{C}^{q+1}(Q)$ be a $m$-dimensional %Riemannian
submanifold minimally immersed into $\mathcal{C}^{q+1}(Q)$ with $Q$ satisfying  the Assumption \ref{assumpt}.
Let $\OM \subset \varphi^{-1}\big((0,a)\times_{y} B_Q(o,R)\big)$ be a connected component with $$R\leq \frac{\pi }{2\sqrt{\Vert G_{-}\Vert_{L^{\infty}([0,R))}}}\cdot $$
Then,
\begin{align}\label{1.1}\lambda^*(\Omega)\geq\frac{1}{a^2}\big(\lambda_{1}(B_g(R))-m\big),
\end{align}
where   $B_g(R)$ is the geodesic  ball of radius $R$  in the model manifold $\mathbb{Q}^{m-1}_g$.
\end{corollary}
We are ready to analyze the spherical case. Although the sphere is  well studied, the
values of the  first  eigenvalue $\lambda_{1}(B_{\mathbb{S}^{m}}(r))$ are pretty
much unknown, with the  exceptions  $\lambda_{1}(B_{\mathbb{S}^{m}}(\pi/2))=m$ and
$\lambda_{1}(B_{\mathbb{S}^{m}}(\pi))=0$. We should mention the estimates for spherical cups \cite{barbosa-do-carmo}, \cite{pinsky},
\cite{Sato} in dimension two,  \cite{friedland-hayman} in
dimension three and  \cite{BB}, \cite{BB2},  \cite{BCG} in all dimensions.
\begin{corollary}\label{coroesfera1}
Let $\varphi :M^m\rightarrow \mathbb{S}^{q+1}=(o, \pi)\times_{\sin y} \mathbb{S}^{q}$ be an $m$-dimensional submanifold minimally immersed into $\mathbb{S}^{q+1}$. Let $\OM \subset \varphi^{-1}((o, r)\times_{\sin y} B_{\mathbb{S}^{q}}(\theta))$, $\theta < \pi/2$ be a connected component.
Then \begin{equation}\label{esferapimedios}
\lambda^*(\Omega) \geq \left\{\begin{array}{ll} \disp \frac{\lambda_1(B_{\mathbb{S}^{m-1}}(\theta))-m}{(\sin r)^{2}} & \qquad {\rm if } \,\,\, r\leq \pi/2,\\
& \\
\lambda_1(B_{\mathbb{S}^{m-1}}(\theta))-m& \qquad {\rm if } \,\,\, r\geq \pi/2.
\end{array} \right.
\end{equation}
 \end{corollary}

\subsection{Essential spectrum}
The ideas developed above can  be applied to study the essential spectrum of $-\Delta$ of   submanifolds properly immersed into the hyperbolic spaces with fairly weak bounds on the mean curvature vector. Via Persson formula (\cite{persson} and \cite[Prop. \!\!3.2]{bmp}), one can express the bottom of the essential spectrum of $-\Delta$ as follows: for every exhaustion  of $M$ by relatively compact open sets $\{K_j\}$ with  Lipschitz boundary,
\begin{equation}\label{persson}
\inf \sigma_\ess(-\Delta) = \lim_{j \ra +\infty} \lambda^*(M \backslash K_j).
\end{equation}
It therefore follows that $-\Delta$ has pure discrete spectrum if and only if
$$
\lim_{j \ra +\infty} \lambda^*(M \backslash K_j)=\infty.
$$
Our next application regards the essential spectrum of graph hypersurfaces of $\mathbb{H}^{q+1}$ whose boundary lies in a relatively compact region of $\mathbb{H}^q_\infty$, the boundary at infinity of $\mathbb{H}^{q+1}$.

\begin{corollary}\label{cor_esshyperb}
Consider the upper half-space model of the hyperbolic space $\mathbb{H}^{q+1}$, $q \ge 2$, with coordinates $(x_0, x_1, \ldots , x_q)= (x_0, \bar x)$ and metric
$$
\metric = \frac{1}{x_{0}^2}\Big( \di x_0^2 + \di x_1^2 + \ldots + \di x_q^2\Big),
$$
and let $\mathbb{H}^q_\infty$ be its boundary at infinity, with chart $\bar x$. Consider a hypersurface without boundary $\varphi : M^{q} \ra \mathbb{H}^{q+1}$ that can be written as the graph of a function $u$ over a relatively compact, open set $W \subseteq \mathbb{H}^q_\infty$, and denote with $H(\bar x)$ its mean curvature. For $z>0$, define
$$
H_z = \sup\left\{ \big|H(\bar x)\big| \, : \, \bar x \in W, \, u(\bar x) = z \right\}
$$
If
\begin{equation}\label{Hz}
\lim_{z \ra 0} z^2 H_z = 0,
\end{equation}
then $M$ has pure discrete spectrum.
\end{corollary}
\begin{proof}
Setting $y = \log x_0$, we can rewrite the metric on $\mathbb{H}^{q+1}$ as the one of the warped product $\erre \times_{e^y} \erre^q$. In our assumptions, since $M$ has no boundary and is a graph over $W$ it holds $y(\varphi(x)) \ra -\infty$ as $x$ diverges in $M^q$. We identify the factor $\erre^q$ in the warped product structure with $\mathbb{H}^q_\infty$ endowed with the Euclidean metric, we fix an origin $o \in \mathbb{H}^q_\infty$ and we let $R$ be large enough that $W \subset B_{\erre^q}(o,R)$. Let $\{z_j\} \downarrow 0^+$ be a chosen sequence, set $\beta_j = \log z_j \downarrow -\infty$ and define
$$
K_j = \varphi^{-1}\big( (\beta_j, +\infty) \times W\big), \qquad \Omega_j = M \backslash K_j.
$$
In our assumptions, $K_j$ is relatively compact for every $j$ and $\{K_j\}$ is a smooth exhaustion of $M$. Consider a positive first eigenfunction $v$ of the geodesic ball $B_{\mathbb{R}^{q-1}}(o,2R)$, with the normalization $\|v\|_{L^\infty}=1$. Define $F\colon (-\infty, \beta_j)\times_{e^{y}} B_{\mathbb{R}^q}(o, 2R)\to \mathbb{R}$ as
$$
F(y, p)=e^{y}\cdot h(p),
$$
where $h(p)=v(\rho_{_{\mathbb{R}^q}}(p))$.
By Theorem \ref{thmBM1} and formula \eqref{eqBF3},
\begin{eqnarray}
\lambda^{\ast}(M\setminus K_j)&\geq & \inf_{M\setminus K_j}\frac{-\Delta (F\circ \varphi)}{F\circ \varphi} \nonumber \\
& =&  \inf_{M\setminus K_j}-\frac{1}{F\circ \varphi}\left[\sum_{i=1}^{q}\Hess_{\mathbb{H}^{q+1}}F (\varphi
(x))\,(e_{i},e_{i}) + q\langle \nabla F,H\rangle\right].\label{eqBF34}\nonumber
\end{eqnarray}
The proof of Theorem \ref{thm2}, in particular inequality \eqref{bgr}, show that, for $x \in M\backslash K_j$,
$$
-\frac{1}{F\circ \varphi}\sum_{i=1}^{q}\Hess_{\mathbb{H}^{q+1}}F (\varphi
(x))\,(e_{i},e_{i}) \geq \frac{\lambda_1(B_{\mathbb{R}^{q-1}}(o,2R))}{e^{2y(x)}} -q,
$$
therefore, on $ M\setminus K_j$,
$$
-\frac{\Delta (F\circ
 \varphi) }{F\circ \varphi}(x)\geq \frac{\lambda_1(B_{\mathbb{R}^{q-1}}(o,2R))}{e^{2y(x)}}
-q - q \, \vert H\vert  \, \frac{\vert \nabla F\vert }{F} (\varphi (x)).
$$
On the other hand, $\nabla F= F \,\nabla y + e^{y}\nabla h$ and thus $\vert \nabla F \vert/F \leq 1 +  \vert \nabla h\vert/h$. Since $1\geq h>0$ on $\overline{B_{\mathbb{R}^{q-1}}(o,R)}$, we infer that
$$
\sup_{B_{\mathbb{R}^{q-1}}(o,R )}\frac{\vert \nabla F\vert }{F}\leq C(R),
$$
where
$$
C(R)=1+ \sup_{B_{\mathbb{R}^{q-1}}(o,R)}\frac{\vert \nabla h\vert}{h}>0.
$$
From the above, we have
\begin{equation}
%\begin{array}{lcl}
\disp \lambda^{\ast}(M\setminus K_j) \geq \inf_{M\backslash K_j} \left[\frac{\lambda_1(B_{\mathbb{R}^{q-1}}(o,2R))- qC(R)|H(x)|e^{2y(x)} - qe^{2y(x)}}{e^{2y(x)}}\right].
%\\[0.2cm]
%& \ge & \inf_{M\backslash K} \left[F(\varphi(x))^{-2}\Big(\lambda_1(B_{\mathbb{R}^{q-1}}(o,2R))- qC(R)\frac{|H_{z(x)}|}{|F(\varphi(x))|^2} - qF(\varphi(x))^{-2}\Big)\right] \\[0.2cm]\disp z_j^{-2}\Big(\lambda_1(B_{\mathbb{R}^{q-1}}(o,R+\epsilon ))- qC(\epsilon, R)z_j^2H_{z_j} - qz_j^2\Big)\nonumber .
\end{equation}
In our assumptions, on $M\backslash K_j$,
$$
|H(x)|e^{2y(x)} \le H_{x_0(x)}e^{2y(x)} = H_{x_0(x)}\big[x_0(x)\big]^2.
$$
By \eqref{Hz}, this latter goes to zero uniformly for $x \in M\backslash K_j$ and divergent $j$. In particular, for each fixed $\eps>0$, there exists $j_\eps$ large such that, for $j \ge j_\eps$, $|H(x)|e^{2y(x)} \le \eps$ on $M\backslash K_j$. It therefore follows that, for $j$ large enough,
$$
\disp \lambda^{\ast}(M\setminus K_j) \geq \inf_{M\backslash K_j} \left[\frac{\lambda_1(B_{\mathbb{R}^{q-1}}(o,2R))- qC(R)\eps - qx_0(x)^2}{x_0(x)^2}\right].
$$
Choosing $\eps$ sufficiently small, letting $j \ra +\infty$ and using that $x_0(x)^2 \le e^{2\beta_j} \ra 0^+$ for $x \in M\backslash K_j$ and divergent $j$, we deduce that $\lambda^\ast(M \backslash K_j) \ra +\infty$, and the claim follows by Persson formula.
\end{proof}
To conclude, we consider the essential spectrum of submanifolds satisfying some strong non-properness assumption. This includes submanifolds with bounded image immersed in a complete manifold. We begin with recalling the following

 \begin{definition} Let $M$, $W$ be Riemannian manifolds and let $\varphi\colon M\to W$  be an isometric  immersion.  The limit set of $\varphi $, denoted by $\lim \varphi $, is a closed set defined as follows
\begin{align}\label{definition1}
\lim \varphi &  = \big\{p\in  W ; \;\exists \,\{p_{k}\}\subset M ,\,{\rm dist }_{M}(o,p_{k})\rightarrow \infty \; {\rm and} \;{\rm dist}_{W}(p, \varphi (p_{k}))\rightarrow 0\big\}.\nonumber
\end{align}
\end{definition}
Observe that: \begin{itemize}\item  An isometric immersion  $\varphi\colon M\to W$ is proper if and only if $\lim \varphi =\emptyset$. \item
The closure of the set   $ \varphi^{-1}[W\setminus T_{\epsilon}(\lim \varphi)]$ may not be a  compact subset of $M$. Here $T_{\epsilon}(\lim \varphi) =\{ y\in W\colon {\rm dist}_{_{W}}(y, \lim \varphi)<\epsilon\} $ is the $\epsilon$-tubular neighborhood of  $\lim \varphi$.
\end{itemize}
\begin{definition}An isometric immersion  $\varphi\colon M\to W$  is strongly non-proper if for all $\epsilon >0$ the closed subset $\varphi^{-1}(W\setminus T_{\epsilon}\lim \varphi)$ is compact in $M$.
\end{definition}
%Observe that if $\varphi\colon M\to W$  is strongly non-proper then  $\varphi^{-1}\left[N \setminus T_{\epsilon}(\lim \varphi)\right]$ is compact in $M$, where $T_{\epsilon}(\lim \varphi)=\{ y\in W\colon {\rm dist}_{W}(y, \lim \varphi)<\epsilon\}$.
%
\begin{remark}
\emph{
A strongly non-proper immersions is not necessarily bounded: for example, the graph immersion $\varphi : B_1(0) \backslash \{0\} \subset \erre^m \ra \erre^m \times \erre$ given by
$$
\varphi(x) = (w,z) = \left(x, \frac{1-r(x)}{r(x)} \sin\big( r(x)(1-r(x) \big)\right)
$$
is strongly non-proper, and $\lim \varphi = \{ w= 0\} \cup \{ r(w)=1, \, z= 0\}$.
}
\end{remark}

\begin{corollary}\label{thm33}
  Let $\varphi \colon \!M^m\rightarrow N^n\times_f Q^q $ be a strongly non-proper minimal submanifold. Suppose that $Q$ satisfies Assumption \ref{assumpt}. Assume in addition  that
  the warping function $f$ satisfies $ \inf_{N}f>c_1>0$, $\sup_{N} \vert \nabla f\vert\leq c_{2}<\infty $ and
  $$
  \Hess_N f(\cdot,\cdot)-\frac{\left|\nabla^N f\right|_N^2}{f}\left\langle ,\right\rangle_N\leq 0.$$
Then, if $\lim \varphi \subset  N\times_{f}\{o\}$, the spectrum of $M$ is discrete.
\end{corollary}
\begin{proof}
Let $T_{j}(N)=N\times_f B_{Q}(o, 1/j)$, for $j$ large enough that $B_{Q}(o, 2/j) \Subset M$ is a regular, convex ball. Let $K_{j}=\varphi^{-1}\left[(N\times_{f}Q)\setminus T_{j}(N)\right]$ be an exhaustion of $M$ by relatively compact, open sets. Note that $\varphi (M\setminus K_j)\subset T_{j}(N)$.
We now proceed as in the proof of Corollary \ref{cor_esshyperb}. Define $F = f(p)v_{j}(\rho_{_Q}(q))$, where $v_j$ is the first eigenfunction of $B_g(2/j) \subset \mathbb{Q}^{m-n}_g$, normalized according to $\|v_j\|_{L^\infty} = 1$, and note that
$$
\Vert \nabla \log F\Vert_{_{L^\infty(T_j(N))}} \le \|\nabla \log f\| + \|\nabla \log v_j\| \le \frac{c_2}{c_1} + \|\nabla \log v_j\|.
$$
By gradient estimates (see for instance,  \cite[Thm. 6.1]{Peter Li}.)
$$
\|\nabla \log v_j\|_{_{L^\infty(T_j)}} = \left\| \frac{v_j'}{v_j} \right\|_{L^\infty([0,j])} \le C\cdot j,
$$
for some absolute constant $C>0$, and so $\|\nabla \log F\|_{_{L^\infty(T_j)}} \le Cj$. Using formula \eqref{eqBF34} and proceeding as in the proof of Corollary \ref{cor_esshyperb}, we have that
\begin{align*}
\lambda^{\ast}(M\setminus K_j)&\geq   \inf_{p\in N} \left(\frac{\lambda_{1}(B_g(2/j)) - m\,\vert \nabla^N f\vert_N^2(p)}{\vert f(p)\vert^2}\right)\\\\ &\quad  - m\|H\|_{L^\infty(M)} \|\nabla \log F\|_{L^\infty(T_j)}.
\end{align*}
Since
$$
\frac{\lambda_{1}(B_g(2/j)) - m\,\vert \nabla^N f\vert_N^2(p)}{\vert f(p)\vert^2}\geq \frac{\lambda_{1}(B_g(2/j)) - mc_{2}^{2}}{c_{1}^{2}},
$$
we deduce
$$
\lambda^{\ast}(M\setminus K_j)\geq \frac{\lambda_{1}(B_g(2/j)) - mc_{2}^{2}}{c_{1}^{2}} - m\|H\|_{_{L^\infty(M)}} Cj.
$$
Taking into account the standard asymptotic $\lambda_1(B_g(2/j)) \sim Cj^2$, for some $C>0$, we conclude that
$$
\lim_{j\ra +\infty} \lambda^{\ast}(M\setminus K_j) = +\infty,
$$
and the thesis follows by Persson formula.
\end{proof}

\vspace{0.5cm}

\noindent \textbf{Acknowledgements:} The first author was partially supported by CNPq, grant \#  301041/2009-1.
The second author was partially supported by  MICINN project MTM2009-10418 and Fundaci\'{o}n S\'{e}neca
project 04540/GERM/06, Spain, by  the Inter-university Cooperation Programme Spanish-Brazilian project PHB2010-0137 and by a research training grant within the framework of the programme
Research Training in Excellence Groups GERM by Universidad de Murcia.
This work was developed during the third author's visiting period at the  Universidade Federal do Cear\'{a}-UFC, Fortaleza-Brazil. He wishes to thank the Mathematics Department  for the warm hospitality and for the delightful research environment. The fourth author was supported by FPI Grant BES-2010-036829 and by was partially supported by  MICINN project MTM2009-10418 and Fundaci\'{o}n S\'{e}neca
project 04540/GERM/06, Spain.
This research is a result of the activity developed within the framework of the Programme in Support of Excellence Groups of the Regi\'{o}n de Murcia, Spain, by Fundaci\'{o}n S\'{e}neca, Regional Agency
for Science and Technology (Regional Plan for Science and Technology 2007-2010).

\bibliographystyle{amsplain}

\end{document}